\documentclass[12pt, twoside]{article} 

\usepackage[utf8]{inputenc} 
\usepackage{latexsym,amsmath,amssymb,textcomp,amsthm,times}

\usepackage{geometry} 
\usepackage{tikz}
\usepackage{graphicx} 

\usepackage{calrsfs}
\usepackage[english]{babel}
\usepackage{booktabs} 
\usepackage{array} 
\usepackage{paralist} 
\usepackage{verbatim} 
\usepackage{subfig} 
\usepackage{pgf}
\usepackage{fancyhdr} 
\usepackage{hyperref}
\makeatletter
\def\author#1{\gdef\autrun{\def\and{\unskip, }#1}\gdef\@author{#1}}

\makeatother

\usepackage{sectsty}
\allsectionsfont{\sffamily\mdseries\upshape} 
\usepackage{amsmath,amsfonts,amssymb,authblk}

\usepackage[nottoc,notlof,notlot]{tocbibind} 
\usepackage[titles,subfigure]{tocloft} 

\usetikzlibrary{arrows,automata}

\newtheorem{prop}{Proposition}[subsection]
\newtheorem{lem}[prop]{Lemma}

\newtheorem{theo}{Theorem}

\newtheorem{defin}{Definition}

\newcommand{\cg}{[\kern-0.15em [}
\newcommand{\cd}{]\kern-0.15em]}

\newcommand{\E}{\mathbb{E}}
\newcommand{\Prob}{\mathbb{P}}

\newcommand{\N}{\mathbf{N}}
\newcommand{\Z}{\mathbf{Z}}

\newcommand{\dd}{\mathrm{d}}

\title{Macroscopic cycles for the interchange and quantum Heisenberg models on random regular graphs}
\author{R\'emy Poudevigne-\--Auboiron \thanks{email: rp698@cam.ac.uk}}
\affil{University of Cambridge}
\date{} 

\begin{document}
\maketitle
\begin{abstract}
The interchange process is a random permutation model that was introduced as a way to study the quantum Heisenberg model. For this model, progress had been made on some specific graphs: trees, the hypercube, the Hamming graph, the complete graph and the two block graph. Here we show that for large enough parameters, both the interchange process and the quantum Heisenberg model have macroscopic clusters on random d-regular graphs. Such a result was only known for the complete graph and the two blocks graph.  
\end{abstract}

\section{Introduction and statement of the results}
In this paper we will be interested in two similar models: the interchange model and the quantum Heisenberg model (see \cite{ReviewInterchange} and \cite{ReviewInterchangeUeltschi} for a more thorough presentation of both models). The interchange model was introduced by Harris in \cite{HarrisInterchange} and was first used to study the quantum Heisenberg model by T\'oth in \cite{TothStirring}. Both models can be seen as models of random permutations and it is the point of view we will adopt. The questions that arise when looking at both models concern the cycles of the random permutation: whether two points are in the same cycle, whether there are infinite/macroscopic clusters or only finite small clusters. For both models a simple comparison with percolation tells us that in any dimension there are no infinite cycles for small enough parameters (theorem 6.1 of \cite{ReviewInterchange}). For the quantum model, we have a Mermin-Wagner theorem (\cite{MerminWagner}) which can be interpreted as the absence of infinite cycles in dimension 1 and 2 for any choice of parameters, and we even have polynomial decay of the probability that two vertices are in the same cycle (\cite{PolyDecayQuantumHeis}). We also have that for a slightly different model, the anti-ferromagnetic quantum Heisenberg model, on $\Z^d$ with $d\geq 5$, for large enough weights we have the existence of macroscopic clusters (\cite{DysonLiebSimon}), because of reflection positivity. Surprisingly this is only known for the anti-ferromagnetic model, not for the usual quantum Heisenberg model. On the complete graph, it was shown that for a parameter larger than $1/2$, there are macroscopic cycles for the interchange model and their sizes are given by a Poisson-Dirichlet measure (when the size of the complete graph goes to infinity) in \cite{CompleteSchramm}. Similarily, for the quantum Heisenberg model on the two-block graph (a generalization of the complete graph) it was recently shown that there is a sharp phase transition between finite cycles and macroscopic cycles in \cite{TwoBlockHeisenberg}. On trees, for the interchange model, if the degree is high enough, it is shown that there is a sharp phase transition between finite and infinite cycles (\cite{InterchangeTree1},\cite{InterchangeTree2},\cite{InterchangeTree3}), but we do not have the existence of macroscopic cycles. For the interchange process, the existence of macroscopic clusters has been shown on the complete graph (\cite{InterchangeCompleteGraphTrick}), the Hamming graph (\cite{InterchangeHamming} \cite{InterchangeHamming2}  ) and the hypercube (\cite{InterchangeHypercube} but here the large clusters are only of size $n^{1/2}$, they are not exactly macroscopic) using a similar method. In those case the number of cycles is compared to the number of clusters of percolation through a coupling. Here we will use a different but similar technique (percolation is not involved, but we use a lot of the same estimates) to show the existence of macroscopic clusters on $d$-regular graphs for both the interchange process and the quantum Heisenberg model for large enough parameters. For the quantum Heisenberg model, we use the log-convexity of the partition function to get a precise result. Because we do not have an equivalent for the interchange process, our result needs some averaging.\\
Now we will define precisely the models before stating our results. Here we will use the random permutation representation of both models and for our purpose it is actually simpler to start by the interchange model and then introduce the permutation representation of the Heisenberg model. \\
As both models deal with transposition we need to define a few notations first. For any pair of vertices $(x,y)$, we will call $\tau_{\{x,y\}}$ the transposition that exchanges $x$ and $y$ (identity if $x=y$). For any two permutations $\sigma_1$ and $\sigma_2$, we will call $\sigma_1.\sigma_2$ the composition of the two permutations ($\sigma_1.\sigma_2(x):=\sigma_1(\sigma_2(x))$). For any permutation $\sigma$ on a finite set, $N(\sigma)$ will be the number of cycles of $\sigma$.\\
The interchange model is defined as follows: we take a finite graph $(V,E)$ and define a continuous-time jump process $(X_t)_{t\in[0,\infty)}$ on the set of permutations of $V$ where $X$ jumps from $\sigma$ to $\tau_{x,y}.\sigma$ at a rate $W_{\{x,y\}}$.  Now, for the quantum model, we need an extra parameter $\theta\in(0,\infty)$, and we can define the partition functions $Z_{\theta}$ by:
\begin{equation}
Z_{\theta}(t):=\E\left(\theta^{N(X_t)}\right).
\end{equation}
From this we can define the probability $\Prob_{\theta,t}$ by:
\begin{equation}
\Prob_{\theta,t}(A):=\frac{1}{Z_{\theta}(t)}\E\left(\theta^{X_t}1_{A}\right).
\end{equation}
For any choice of $\theta$ and $t$, we will call $\E_{\theta,t}$ the expectation associated with $\Prob_{\theta,t}$. Even though adding this $\theta$ term may seem to make the interchange model more complicated, when $\theta$ is an integer, the partition function is just the trace of the exponential of a matrix which gives a lot of information. Furthermore, the special case $\theta=2$ corresponds to the quantum Heisenberg model. In all these models, what we will look at is the existence of macroscopic cycles, that is to say cycles with a size equal to a positive fraction of the total size of the underlying graph. This only makes sense in the limit when the size of the graphs go to infinity.\\
We get the existence of such macroscopic cycles on $d$-regular graphs. For any integer $d$, a $d$-regular graph is a graph where all vertices have degree $d$. For given $n$ and $d$, we will call $p_{n,d}$ the uniform probability measure on $d$-regular graphs of size $n$.
We define the following function to simplify the statements of the theorems. 
\begin{defin}
For any $\theta\in (0,\infty)$ and any $d> 2(1+\theta)$, we define $T(\theta,d)$ by:
\begin{equation}
\begin{aligned}
T(\theta,d)=&\frac{2\theta\log(\theta)}{(\theta-1)(d-2(1+\theta))} \text{ if } \theta\not =1\\
T(1,d)=&\frac{2}{d-4} .
\end{aligned}
\end{equation}
The value of $T$ for $\theta=1$ is such that $T$ is continuous.
\end{defin}

We first state our result when $\theta$ is an integer greater or equal to 2 as the result is simpler to understand. In this case we show that for $d> 2(1+\theta)$, there are macroscopic clusters after time $T(\theta,d)$.

\begin{theo}\label{theo:bestquantum}
Set $\theta\in \N\backslash\{0,1\}$ and $d\in\N\cap(2\theta+1,\infty)$. For any $\eta>0$, let $A_{\eta}$ be the event that there is a macroscopic cycle of size larger than $\eta|V|$. For all $t> T(\theta,d)$, there exists $\epsilon,\eta>0$ such that:
\begin{equation}
p_{d,n}\left(\Prob_{\theta,t}(A_{\eta}) \geq \epsilon \right)\xrightarrow[n\rightarrow\infty]{} 1. 
\end{equation}
\end{theo}

In the other cases, the result we get is not as straightforward. The idea is that instead of having macroscopic clusters for a parameter larger than $T(\theta,d)$, we get that in any interval $[a,a+T(\theta,d)]$ there are parameters (that may depend on $n$) where we have macroscopic clusters.

\begin{theo}\label{theo:weakquantum}
Set $\theta\in(0,\infty)$. For any $\eta>0$, let $A_{\eta}$ be the event that there is a macroscopic cycle of size larger than $\eta|V|$. For any $d> 2(1+\theta)$ and $\epsilon<\frac{d-2}{2}$, there exists $\eta>0$ such that for any $b\geq a\geq 0$:
\begin{equation}
p_{d,n}\left(\int\limits_{t=a}^{b}\Prob_{\theta,t}(A_{\eta}) \dd t \geq \frac{(b-a)\left(\frac{d}{2}-(\theta+1)(1+\epsilon)\right)-T(\theta,d)\left(\frac{d}{2}-(1+\theta)\right)}{(1+\theta)\left(\frac{d}{2}-(1+\epsilon)\right)}\right)\xrightarrow[n\rightarrow +\infty]{} 1.
\end{equation} 
For any $\delta>0$, there exists $\eta,c>0$ such that for any $a\geq 0$:
\begin{equation}
p_{d,n}\left(\exists t\in [a,a+T(\theta,d)+\delta],\text{ such that } \Prob_{\theta,t}(A_{\eta})  
\geq c\right)\xrightarrow[n\rightarrow +\infty]{} 1.
\end{equation}
\end{theo}

\section{A preliminary result for $d$-regular graphs}

A practical way to study $d$-regular graphs is the random pairing. For any $n>0$, such that $nd$ is even, we take a uniform random pairing of $\{1,\dots,d\}\times\{1,\dots,n\}$. Then, for any $i,j\in V$ we put as many vertices between $i$ and $j$ as their are pairs of the form $\{(i,a),(j,b)\}$ in our random pairing. Then if we condition on having no multiple edges or loops (edges from a vertex to itself) we get a $d$-regular graph chosen uniformly at random. A nice result of \cite{RandomPairing} tells us that for any $d$, as $n$ goes to infinity, the probability that there are no loops or multiple edges goes to $e^{-\frac{d^2-1}{4}}$. This means that if we can get estimates for the random pairing, we can also easily get them for $d$-regular graphs. In this section, we want to show that with probability going to one as $n$ goes to infinity, all the subsets of a $d$-regular graph up to a given size ($\eta n$ for some $\eta>0$) do not have much more edges than vertices. This will be the main ingredient for our results.

\begin{lem}
Set $\epsilon>0$. For the random pairing model with $nd$ vertices, for a given subset $V$ of size $dv$ with $\frac{v}{n}\leq \frac{\epsilon}{2d},\frac{1}{2}$ we have:
\begin{equation}
\Prob(|E_V|\geq (1+\epsilon) v ) \leq e^{-\log(\frac{\epsilon}{2 d} \frac{n}{v})\left(1+\frac{\epsilon}{2}\right)v}.
\end{equation}
\end{lem}
\begin{proof}
The idea is the following. At every step we pick a vertex of $V$ that has not been paired. We pick uniformly at random (among the non paired vertices) which vertex it is paired with. It is simple to see that at every step the probability that the vertex is paired with a vertex inside $V$ is smaller or equal to $\frac{dv}{dn}=\frac{v}{n}$. This means that $|E_V|$ is stochasticaly dominated by a binomial of parameters $\left(dv,\frac{v}{n}\right)$. We therefore have for any $\lambda>0$:
\begin{equation}
\E\left(e^{\lambda|E_V|}\right)\leq \left(1+(e^{\lambda}-1)\frac{v}{n}\right)^{dv} \leq e^{dv (e^{\lambda}-1)\frac{v}{n}}.
\end{equation}
This means that for any $\alpha,\lambda >0$:
\begin{equation}
\Prob(|E_v|\geq \alpha v) \leq e^{dv (e^{\lambda}-1)\frac{v}{n}- \lambda\alpha v}
=e^{-\lambda v \left(\alpha-d\frac{e^{\lambda}-1}{\lambda}\frac{v}{n}\right)}
\leq e^{-\lambda v \left(\alpha-de^{\lambda}\frac{v}{n}\right)}.
\end{equation}
By taking $\alpha=1+\epsilon$ and $\lambda=\log(\frac{\epsilon}{2 d}\frac{n}{v}) \geq 0$, we get:
\begin{equation}
\Prob(|E_v|\geq (1+\epsilon) v) \leq e^{-\log(\frac{\epsilon}{2 d} \frac{n}{v})\left(1+\frac{\epsilon}{2}\right)v}.
\end{equation}
\end{proof}

\begin{lem}\label{lem:dregular}
Let $d\in \N \backslash \{0,1,2 \}$ be an integer larger than 2. Let $G_{n,d}$ be a random $d$-regular graph chosen according to the measure $p_{n,d}$. For any $\epsilon>0$, there exists $\eta>0$ such that the probability that there exists a subset $V$ of size smaller smaller than $\eta n$ that contains more than $(1+\epsilon) |V|$ edges goes to 0 as $n$ goes to infinity.
\end{lem}
\begin{proof}
By Stirling formula, there exists a constant $C$ such that, for $k\in (1,n)\cap \N$:
\begin{equation}
\binom{n}{k}\leq C \left(\frac{n}{k}\right)^{k}\left(\frac{n}{n-k}\right)^{n-k}.
\end{equation}
The random pairing model has a probability of giving a $d$-regular graph uniformly bounded from below in $n$ this means that there exists a constant $C^{'} $ such that:
\begin{equation}
\begin{aligned}
P_d(n)\leq& \sum\limits_{k=2}^{\eta n} C^{'} \left(\frac{n}{k}\right)^{k}\left(\frac{n}{n-k}\right)^{n-k} e^{-\log(\frac{\epsilon}{2 d} \frac{n}{k})\left(1+\frac{\epsilon}{2}\right)k}\\
=&C^{'}\sum\limits_{k=2}^{\eta n} e^{k\log\left(\frac{n}{k}\right) + (n-k)\log\left(\frac{n}{n-k}\right)-\left(1+\frac{\epsilon}{2}\right)k\log(\frac{\epsilon}{2 d})-\left(1+\frac{\epsilon}{2}\right)\log( \frac{n}{k})}\\
=&C^{'}\sum\limits_{k=2}^{\eta n} e^{k\log\left(\frac{n}{k}\right) + (n-k)\log\left(1+\frac{k}{n-k}\right)-\left(1+\frac{\epsilon}{2}\right) k\log(\frac{\epsilon}{2 d})-\left(1+\frac{\epsilon}{2}\right)\log( \frac{n}{k})}\\
\leq&C^{'}\sum\limits_{k=2}^{\eta n} e^{k\log\left(\frac{n}{k}\right) + (n-k)\frac{k}{n-k}-\left(1+\frac{\epsilon}{2}\right)k\log(\frac{\epsilon}{2 d})-\left(1+\frac{\epsilon}{2}\right)k\log( \frac{n}{k})} \text{ for } \eta \text{ small enough}\\
=&C^{'}\sum\limits_{k=2}^{\eta n} e^{-k\left( \frac{\epsilon}{2}\log( \frac{n}{k}) -1+\left(1+\frac{\epsilon}{2}\right)\log(\frac{\epsilon}{2 d})\right)}\\
\leq&C^{'}\sum\limits_{k=2}^{\eta n} e^{-k\frac{\epsilon}{4}\log( \frac{n}{k})}\text{ for } \eta \text{ small enough}.
\end{aligned}
\end{equation}
From this we have the desired result.
\end{proof}

\section{Interchange model}

The proofs of our results are similar but unfortunately we need to apply the same methods to different quantities so we need to split the proofs between the interchange model and the quantum case. In both cases, we need the following definition.
\begin{defin}
For any finite graph $(V,E)$ and any permutation $\sigma$ of $V$ we define $E_=^{\sigma}$ as the set of edges of $E$ with both endpoints in the same cycle of $\sigma$. That is to say if $(C_i)_{i\in\{1,\dots,k\}}$ are the cycles of $\sigma$ then:
\begin{equation}
E_=^{\sigma}:=\bigcup\limits_{i\in\{1,\dots,k\}}\{\{x,y\}\in E,\text{ such that } x,y\in C_i\}.
\end{equation}
\end{defin}
The reason we look at this set is because it characterizes whether the number of cycles of $X_t$ will tend to increase or decrease. More precisely, we have the following result.
\begin{lem}
For any finite graph $(V,E)$, any permutation $\sigma$ and any edge $\{x,y\}$, we have:
\begin{equation}
N\left(\tau_{x,y}.\sigma\right)=
\left\{
\begin{matrix}
N(\sigma)+1 \text{ if } \{x,y\}\in E_=^{\sigma}\\
N(\sigma)-1 \text{ if } \{x,y\}\not\in E_=^{\sigma}
\end{matrix}
\right.
\end{equation}
\end{lem}
\begin{proof}
It is easy to see that the cycles of $\sigma$ and $\tau_{x,y}.\sigma$ are the same except for those that contain $x$ or $y$. If $\{x,y\}\in E_=^{\sigma}$ then $x$ and $y$ are in the same cycle in sigma but this cycle is split in two in $\tau_{x,y}.\sigma$ so we get $N\left(\tau_{x,y}.\sigma\right)=N(\sigma)+1$. Conversely, if $\{x,y\}\not\in E_=^{\sigma}$ then $x$ and $y$ are in two distinct cycles in sigma but these cycles are merged in $\tau_{x,y}.\sigma$ so we get $N\left(\tau_{x,y}.\sigma\right)=N(\sigma)-1$.
\end{proof}

From this we have that if $|E_=^{X_t}|$ is small then the number of cycles of $X$ will tend to decrease with time. Furthermore, if there is no macroscopic cycle, our result on $d$-regular graphs implies that $|E_=^{X_t}|$ will be small and therefore the number of cycles will tend to decrease. Then the arguments boils down to saying that the number of cycles is between $1$ and $n$ so it can only decrease for so long and this means that there has to be a macroscopic cluster after a reasonable amount of time. We start by proving our result for the interchange model.\\

\begin{lem}\label{lem:stirringbound}
Let $(V,E)$ be a finite graph with the following property: there exists $\epsilon,\eta>0$ such that for any subset $S$ of $V$ if $|S|\leq \eta |V|$ then the number of edges inside $S$ (that we will note $E_S$) satisfies $|E_S|\leq (1+\epsilon)|S|$. Then if $A_{\eta}$ is the event that there is a cycle larger than $\eta |V|$ we have:
\begin{equation}
2(|E|-(1+\epsilon)|V|)\frac{1}{s}\int\limits_{t=a}^{a+s}\Prob_t(A_{\eta}) \dd t \geq |E|-\left(2(1+\epsilon)+\frac{1}{s}\right)|V|.
\end{equation}
\end{lem}
\begin{proof}
\begin{equation}
\begin{aligned}
\frac{\partial}{\partial t} \E(N(X_t)) =& \E_t\left(  2|E_=^{X_t}| - |E| \right)\\
=& \E_t\left( (2|E_=^{X_t}| - |E|)1_{A_{\eta}} \right) + \E_t\left( (2|E_=^{X_t}| - |E|)(1-1_{A_{\eta}}) \right)\\
\leq& |E| \Prob_t(A_{\eta}) + \E_t\left( (2(1+\epsilon)|V|-|E|)(1-1_{A_{\eta}}) \right)\\
=& |E| \Prob_t(A_{\eta})  + (2(1+\epsilon)|V|-|E|)\E_t\left( (1-1_{A_{\eta}}) \right)\\
=& (2(1+\epsilon)|V|-|E|) + 2(|E|-(1+\epsilon)|V|) \Prob_t(A_{\eta}). 
\end{aligned}
\end{equation} 
Then we have that $0\leq N(X_t) \leq |V|$. Therefore for any $a\geq 0$ and $s>0$:
\begin{equation}
\frac{1}{s}\int\limits_{t=a}^{a+s} \frac{\partial}{\partial t} \E(N(X_t)) \dd t = \frac{\E(N(X_{a+s}))-\E(N(X_a))}{s} \geq -\frac{|V|}{s}.
\end{equation}
By putting the two inequalities together, we get:
\begin{equation}
\frac{1}{s}\int\limits_{t=a}^{a+s} (2(1+\epsilon)|V|-|E|) + 2(|E|-(1+\epsilon)|V|) \Prob_t(A_{\eta}) \dd t \geq -\frac{|V|}{s}.
\end{equation}
Therefore:
\begin{equation}
(2(1+\epsilon)|V|-|E|) + 2(|E|-(1+\epsilon)|V|)\frac{1}{s}\int\limits_{t=a}^{a+s}\Prob_t(A_{\eta}) \dd t \geq -\frac{|V|}{s}.
\end{equation}
Finally:
\begin{equation}
2(|E|-(1+\epsilon)|V|)\frac{1}{s}\int\limits_{t=a}^{a+s}\Prob_t(A_{\eta}) \dd t \geq |E|-\left(2(1+\epsilon)+\frac{1}{s}\right)|V|.
\end{equation}
\end{proof}
We see that if the set of edges is large enough compared to the set of vertices, this gives us a lower bound on the average of $\Prob_t(A_{\eta})$ on intervals. Unfortunately, the interval need to be large enough for the rightmost term to be positive so we only get result on average for some times, not for a given time. This result is similar to the one obtained on the hypercube in \cite{InterchangeHypercube}. Now we have everything we need to get our main result for $d$-regular graphs in the case $\theta=1$.
\begin{lem}\label{lem:stirring}
Set $d\geq 5$. For any $\eta>0$ and any finite graph $(V,E)$, let $A_{\eta}$ be the event that there is a macroscopic cycle of size larger than $\eta|V|$. For any $\epsilon\in\left(0,\frac{d-2}{2}\right)$, there exists $\eta>0$ such that for any $a\geq 0$ and any $s>0$:
\begin{equation}
p_{d,n}\left(\frac{1}{s}\int\limits_{t=a}^{a+s}\Prob_t(A_{\eta}) \dd t \geq \frac{d-2\left(2(1+\epsilon)+\frac{1}{s}\right)}{2(d-2(1+\epsilon))}\right)\xrightarrow[n\rightarrow +\infty]{} 1.
\end{equation} 
For any $\delta>0$, there exists $\eta,c>0$ such that for any $a\geq 0$:
\begin{equation}
p_{d,n}\left(\int\limits_{t=a}^{a+\frac{2}{d-4}+\delta}\Prob_t(A_{\eta}) \dd t \geq c\right)\xrightarrow[n\rightarrow +\infty]{} 1.
\end{equation}
\end{lem}
\begin{proof}
By lemma \ref{lem:dregular}, for any $\epsilon>0$, there exists $\eta>0$ such that the probability that a d-regular graph with $n$-vertices has no subset $S$ of size smaller than $\eta n$ with more than $(1+\epsilon)|S|$ edges goes to $1$ as $n$ goes to infinity. Let $(V,E)$ be such a graph, by lemma \ref{lem:stirringbound} we have:
\begin{equation}
2(|E|-(1+\epsilon)|V|)\frac{1}{s}\int\limits_{t=a}^{a+s}\Prob_t(A_{\eta}) \dd t \geq |E|-\left(2(1+\epsilon)+\frac{1}{s}\right)|V|.
\end{equation}
For our graph, we have $|V|=n$ and $|E|=\frac{nd}{2}$ so:
\begin{equation}
2\left(\frac{nd}{2}-(1+\epsilon)n\right)\frac{1}{s}\int\limits_{t=a}^{a+s}\Prob_t(A_{\eta}) \dd t \geq \frac{nd}{2}-\left(2(1+\epsilon)+\frac{1}{s}\right)n.
\end{equation}
If $\epsilon<\frac{d-2}{2}$ we get:
\begin{equation}
\frac{1}{s}\int\limits_{t=a}^{a+s}\Prob_t(A_{\eta}) \dd t \geq \frac{d-2\left(2(1+\epsilon)+\frac{1}{s}\right)}{2(d-2(1+\epsilon))}.
\end{equation} 
For any $\delta>0$, by taking $\epsilon= \frac{(d-4)(1+\delta(d-4))}{8(2+\delta(d-4))}<\frac{d-4}{8}$ and the corresponding $\eta$ we get:
\begin{equation}
\begin{aligned}
\int\limits_{t=a}^{a+\frac{2}{d-4}+\delta}\Prob_t(A_{\eta}) \dd t 
\geq& \left(\frac{2}{d-4}+\delta\right)\frac{d-4-\left(\frac{(d-4)(1+\delta(d-4))}{2(2+\delta(d-4))}+\frac{d-4}{2+\delta(d-4)}\right)}{2(d-2(1+\frac{d-4}{8}))}\\
=&\frac{(d-4)(1+\delta(d-4))}{8+\frac{7}{2}(d-4)}\\
>&0.
\end{aligned}
\end{equation}
\end{proof}

Now we can look at the quantum case. The ideas are the same but instead of looking at the expectation of the number of cycles we look at the partition function.

\begin{lem}\label{lem:quantumbound}
Set $\epsilon,\theta,\eta>0$, with $\theta\not = 1$. Let $(V,E)$ be a graph such that any subset $S$ of $V$ of size smaller than $\eta n$ has less than $(1+\epsilon)|S|$ edges. Let $A_{\eta}$ be the event that $X_t$ has a cycle of size larger than $\eta |V|$. For any $t\geq 0$:
\begin{equation}
\frac{\theta}{\theta-1}\frac{\partial}{\partial t} \log\left(Z_{\theta}(t)\right)
\leq \left((\theta+1)(1+\epsilon)|V|-|E|\right) +(1+\theta)\left(|E|-(1+\epsilon)|V|\right)\Prob_{\theta,t}(A_{\eta}).
\end{equation}
\end{lem}
\begin{proof}
We have:
\begin{equation}
\begin{aligned}
\frac{\partial}{\partial t}Z_{\theta}(t)
=&\sum\limits_{e\in E} \E_{t}\left(\theta^{N(\tau_e.X_t)}-\theta^{N(X_t)}\right)\\
=&\sum\limits_{e\in E} \E_{t}\left(\theta^{N(X_t)}\left(\theta^{N(\tau_e.X_t)-N(X_t)}-1\right)\right)\\
=&\sum\limits_{e\in E} \E_{t}\left(\theta^{N(X_t)}\left((\theta-1)1_{e\in E^{X_t}_=}+\left(\frac{1}{\theta}-1\right)1_{e\not\in E^{X_t}_=}\right)\right)\\
=&\E_{t}\left(\theta^{N(X_t)}\left((\theta-1)|E^{X_t}_=|-\frac{\theta-1}{\theta}(|E|-|E^{X_t}_=|)\right)\right)\\
=&\frac{\theta-1}{\theta}\E_{t}\left(\theta^{N(X_t)}\left((\theta+1)|E^{X_t}_=|-|E|\right)\right)
\end{aligned}
\end{equation}
Now we introduce the event $A_{\eta}$:
\begin{equation}
\begin{aligned}
&\E_{t}\left(\theta^{N(X_t)}\left((\theta+1)|E^{X_t}_=|-|E|\right)\right)\\
=&\E_{t}\left(\theta^{N(X_t)}\left((\theta+1)|E^{X_t}_=|-|E|\right)1_{A_{\eta}}\right)+\E_{t}\left(\theta^{N(X_t)}\left((\theta+1)|E^{X_t}_=|-|E|\right)\left(1-1_{A_{\eta}}\right)\right)\\
\leq&  \E_{t}\left(\theta^{N(X_t)}\theta|E|1_{A_{\eta}}\right)+\E_{t}\left(\theta^{N(X_t)}\left((\theta+1)(1+\epsilon)|V|-|E|\right)\left(1-1_{A_{\eta}}\right)\right)\\
=& Z_{\theta}(t) \theta|E|\Prob_{\theta,t}(A_{\eta})+\left((\theta+1)(1+\epsilon)|V|-|E|\right)Z_{\theta}(t)(1-\Prob_{\theta,t}(A_{\eta}))\\
=& \left((\theta+1)(1+\epsilon)|V|-|E|\right)Z_{\theta}(t) + \left((\theta+1)|E|-(\theta+1)(1+\epsilon)|V|\right)Z_{\theta}(t)\Prob_{\theta,t}(A_{\eta}) .
\end{aligned}
\end{equation}
Thus we have
\begin{equation}
\frac{\theta}{\theta-1}\frac{\partial}{\partial t} \log\left(Z_{\theta}(t)\right)
\leq \left((\theta+1)(1+\epsilon)|V|-|E|\right) +(1+\theta)\left(|E|-(1+\epsilon)|V|\right)\Prob_{\theta,t}(A_{\eta}).
\end{equation}
\end{proof}

The proof of theorem \ref{theo:weakquantum} for general $\theta$ is similar to the proof of lemma \ref{lem:stirring} for the case $\theta=1$.

\begin{proof}[Proof of theorem \ref{theo:weakquantum}]
The case $\theta=1$ is done in lemma \ref{lem:stirring}. We now focus on the case $\theta\not=1$. The first thing to notice is that the number of cycles of $X_t$ is between $0$ and $|V|$ so:
\begin{equation}
0\leq \frac{\theta}{\theta-1}\log(Z_{\theta}(t))\leq |V|\frac{\theta\log(\theta)}{\theta-1}.
\end{equation}
This means that for any $a \geq 0$ and $s>0$ we get:
\begin{equation}
\int\limits_{t=a}^{a+s} \frac{\theta}{\theta-1}\frac{\partial}{\partial t} \log\left(Z_{\theta}(t)\right) \dd t \geq -|V|\frac{\theta\log(\theta)}{\theta-1}.
\end{equation}
Then, by lemma \ref{lem:quantumbound} we have for any $0\leq a< b$:
\begin{equation}
\int\limits_{t=a}^{a+s}\frac{\theta}{\theta-1}\frac{\partial}{\partial t} \log\left(Z_{\theta}(t)\right)
\leq s\left((\theta+1)(1+\epsilon)|V|-|E|\right) +(1+\theta)\left(|E|-(1+\epsilon)|V|\right)\int\limits_{t=a}^{a+s}\Prob_{\theta,t}(A_{\eta})\dd t.
\end{equation}
By putting the two together we get:
\begin{equation}
s\left((\theta+1)(1+\epsilon)|V|-|E|\right) +(1+\theta)\left(|E|-(1+\epsilon)|V|\right)\int\limits_{t=a}^{a+s}\Prob_{\theta,t}(A_{\eta})\dd t
\geq -|V|\frac{\theta\log(\theta)}{\theta-1}.
\end{equation}
By lemma \ref{lem:dregular}, for any $\epsilon>0$, there exists $\eta>0$ such that the probability that a d-regular graph with $n$-vertices has no subset $S$ of size smaller than $\eta n$ with more than $(1+\epsilon)|S|$ edges goes to $1$ as $n$ goes to infinity. Let $(V,E)$ be such a graph, we have $|V|=n$ and $|E|=\frac{nd}{2}$ so:
\begin{equation}
s\left((\theta+1)(1+\epsilon)n-\frac{nd}{2}\right) +(1+\theta)\left(\frac{nd}{2}-(1+\epsilon)n\right)\int\limits_{t=a}^{a+s}\Prob_{\theta,t}(A_{\eta})\dd t
\geq -n\frac{\theta\log(\theta)}{\theta-1}.
\end{equation}
Which is equivalent to:
\begin{equation}
(1+\theta)\left(\frac{d}{2}-(1+\epsilon)\right)\int\limits_{t=a}^{a+s}\Prob_{\theta,t}(A_{\eta})\dd t
\geq s\left(\frac{d}{2}-(\theta+1)(1+\epsilon)\right)-\frac{\theta\log(\theta)}{\theta-1}.
\end{equation}
If $\epsilon<\frac{d-2}{2}$ we get:
\begin{equation}
\int\limits_{t=a}^{a+s}\Prob_{\theta,t}(A_{\eta}) \dd t \geq \frac{s\left(\frac{d}{2}-(\theta+1)(1+\epsilon)\right)-\frac{\theta\log(\theta)}{\theta-1}}{(1+\theta)\left(\frac{d}{2}-(1+\epsilon)\right)}.
\end{equation} 
If $d>2(\theta+1)$, set $s_d:=\frac{\theta\log(\theta)}{(\theta-1)\left(\frac{d}{2}-(\theta+1)\right)}$. For any $\delta>0$, by taking $\epsilon= \frac{\delta}{2(s_d+\delta)}\frac{\frac{d}{2}-(\theta+1)}{\theta+1}$ and the corresponding $\eta$ we get for any $a\geq 0$:
\begin{equation}
\begin{aligned}
\int\limits_{t=a}^{a+s_d+\delta}\Prob_{\theta,t}(A_{\eta}) \dd t 
\geq& \frac{\frac{d}{2}-(\theta+1)}{2(1+\theta)\left(\frac{d}{2}-(1+\epsilon)\right)}\delta\\
>&0.
\end{aligned}
\end{equation}
\end{proof}

Now, we give the main ingredient for the quantum case.

\begin{lem}
We have that  for all $\theta\in\N^*$, $Z_{\theta}(t)$ is $\log$-convex in $t$.
\end{lem}
\begin{proof}
For $\theta=1$ the result is trivial. For $\theta\geq 2$, there exists a symmetric matrix $H_{\theta}$ such that $Z_{\theta}(t)=\text{Tr}\left(e^{t H_\theta}\right)$ (theorem 2.3 of \cite{ReviewInterchangeUeltschi}). This means that is it a linear combination with positive coefficients of exponential functions which is always $\log$-convex.
\end{proof}

This result allows us to get a much more precise statement when $\theta$ is an integer. We now give the proof of theorem \ref{theo:bestquantum}.

\begin{proof}[Proof of theorem \ref{theo:bestquantum}]
We have $1\leq N(X_t) \leq |V|$ and therefore $0\leq \log \left(Z_{\theta}(t)\right)\leq |V|\log(\theta)$. We also have $\log \left(Z_{\theta}(0)\right)=|V|\log(\theta)$. This means that for any $t$:
\begin{equation}
\frac{\frac{1}{|V|}\log(Z_{\theta}(t))-\frac{1}{|V|}\log(Z_{\theta}(0))}{t}\geq -\frac{\log(\theta)}{t}.
\end{equation}
Because the partition function is $\log$-convex, we have:
\begin{equation}
\frac{\frac{1}{|V|}\log(Z_{\theta}(t))-\frac{1}{|V|}\log(Z_{\theta}(0))}{t}\leq \frac{1}{|V|}\frac{\partial}{\partial t} \log\left(Z_{\theta}(t)\right).
\end{equation}
By putting the two together, we get:
\begin{equation}
\frac{1}{|V|}\frac{\partial}{\partial t} \log\left(Z_{\theta}(t)\right) \geq -\frac{\log(\theta)}{t}.
\end{equation}
By lemma \ref{lem:quantumbound} we then have:
\begin{equation}
\frac{1}{|V|}\left((\theta+1)(1+\epsilon)|V|-|E|\right) +(1+\theta)\left(|E|-(1+\epsilon)|V|\right)\Prob_{\theta,t}(A_{\eta}) \geq -\frac{\theta\log(\theta)}{(\theta-1)t}.
\end{equation}
By lemma \ref{lem:dregular}, for any $\epsilon>0$, there exists $\eta>0$ such that the probability that a d-regular graph with $n$-vertices as no subset $S$ of size smaller than $\eta n$ with more than $(1+\epsilon)|S|$ edges goes to $1$ as $n$ goes to infinity. Let $(V,E)$ be such a graph, we have $|V|=n$ and $|E|=\frac{nd}{2}$ so:
\begin{equation}
\left((\theta+1)(1+\epsilon)-\frac{d}{2}\right) +(1+\theta)\left(\frac{d}{2}-(1+\epsilon)\right)\Prob_{\theta,t}(A_{\eta}) \geq -\frac{\theta\log(\theta)}{(\theta-1)t}.
\end{equation}
And therefore:
\begin{equation}
\Prob_{\theta,t}(A_{\eta}) \geq  \frac{\frac{d}{2}-(1+\theta)(1+\epsilon)-\frac{\theta\log(\theta)}{(\theta-1)t}}{(1+\theta)\left(d/2-1-\epsilon \right)}.
\end{equation}
We will write $t_d:=\frac{2 \theta \log(\theta)}{(\theta-1)}\frac{1}{d-2(1+\theta)}$. We see that for any $\delta>0$ there exists $\epsilon>0$ such that:
\begin{equation}
\frac{\frac{d}{2}-(1+\theta)(1+\epsilon)-\frac{\theta\log(\theta)}{(\theta-1)(t_d+\delta)}}{(1+\theta)\left(d/2-1-\epsilon \right)}>0.
\end{equation}
This means that for any $\delta>0$, there exists $\epsilon,\eta>0$ such that
\begin{equation}
\liminf\limits_{n\rightarrow \infty} \Prob_{\theta,t_d+\delta}(A_{\eta}) \geq  \frac{\frac{d}{2}-(1+\theta)(1+\epsilon)-\frac{\theta\log(\theta)}{(\theta-1)(t_d+\delta)}}{(1+\theta)\left(d/2-1-\epsilon \right)}>0.
\end{equation}
\end{proof}

\section{Acknowledgements}
I would like to thank Roland Bauerschmidt for suggesting to look at this problem and both Tyler Helmuth Jakob Bj\"{o}rnberg for helpful discussions. This work was supported by the European Research Council under the European Union's Horizon 2020 research and innovation programme (grant agreement No.~851682 SPINRG).

\bibliographystyle{abbrv}
\bibliography{biblio}
\end{document}